\documentclass[12pt,reqno]{amsart}
\usepackage[margin=1in]{geometry}
\usepackage{amscd,amsfonts,amsmath,amssymb,amsthm,latexsym,mathrsfs,textcomp,verbatim}
\usepackage{accents}
\usepackage{bm}
\usepackage{bbm}
\usepackage{cite}
\usepackage{color}
\usepackage{constants}
\usepackage{csquotes}
\usepackage{enumerate}
\usepackage{etoolbox}
\usepackage{hypbmsec}
\usepackage[breaklinks=true,colorlinks=true,linkcolor=black!25!cyan,citecolor=black!25!cyan,filecolor=cyan,urlcolor=magenta]{hyperref}
\hypersetup{linktocpage}
\usepackage[hiresbb]{graphicx,xcolor}
\usepackage[capitalise]{cleveref}
\usepackage{mathtools}
\usepackage{soul}
\usepackage{tikz}
\usepackage{tikz-cd}
\usetikzlibrary{shapes.geometric}
\usetikzlibrary{shapes.misc}
\usetikzlibrary{positioning}
\allowdisplaybreaks[4]

\newtheorem{theorem}{Theorem}[section]
\newtheorem{corollary}[theorem]{Corollary}

\newtheorem{lemma}[theorem]{Lemma}
\newtheorem{conjecture}[theorem]{Conjecture}

\newtheorem{definition}[theorem]{Definition}
\newtheorem{question*}{Question}
\newtheorem{problem*}{Problem}

\theoremstyle{definition}

\newtheorem{example}[theorem]{Example}

\theoremstyle{remark}
\newtheorem{remark}{Remark}

\numberwithin{equation}{section}

\newcommand{\Z}{\mathbb{Z}}

\newcommand{\R}{\mathbb{R}}
\newcommand{\Q}{\mathbb{Q}}
\newcommand{\F}{\mathbb{F}}

\renewcommand{\pmod}[1]{\, (\mathrm{mod} {\, #1})}

\newcommand{\ord}{\mathop{\mathrm{ord}}}

\renewcommand{\Re}{\mathrm{Re}}

\patchcmd{\section}{\scshape}{\bfseries}{}{}
\makeatletter
\renewcommand{\@secnumfont}{\bfseries}
\makeatother

\makeatletter\newcommand{\tpmod}[1]{{\@displayfalse\pmod{#1}}}
\makeatletter
\@namedef{subjclassname@2020}{\textup{2020} Mathematics Subject Classification}
\makeatother

\begin{document}

\title{A New Aspect of Chebyshev's Bias for Elliptic Curves over Function Fields}

\author{Ikuya Kaneko}
\address{The Division of Physics, Mathematics and Astronomy, California Institute of Technology, 1200 E. California Blvd., Pasadena, CA 91125, USA}
\email{ikuyak@icloud.com}
\urladdr{\href{https://sites.google.com/view/ikuyakaneko/}{https://sites.google.com/view/ikuyakaneko/}}

\author{Shin-ya Koyama}
\address{Department of Biomedical Engineering, Toyo University, 2100 Kujirai, Kawagoe, Saitama, 350-8585, Japan}
\email{koyama@tmtv.ne.jp}
\urladdr{\href{http://www1.tmtv.ne.jp/~koyama/}{http://www1.tmtv.ne.jp/~koyama/}}

\date{\today}



\begin{abstract}
This \textcolor{black}{work considers} the prime number races for non-constant elliptic curves~$E$ over function fields. We prove that if $\mathrm{rank}(E) > 0$, then there exist Chebyshev biases~towards being negative, and otherwise there exist Chebyshev biases towards being positive. \textcolor{black}{The main innovation entails} the convergence of the \textcolor{black}{partial} Euler product at the centre \textcolor{black}{that~follows~from} the Deep Riemann Hypothesis over function fields.
\end{abstract}

\maketitle

\section{Introduction}\label{Introduction}
In 1853, Chebyshev noticed in a letter to Fuss that primes congruent to $3$ modulo $4$ seem to predominate over those congruent to $1$ modulo $4$. If $\pi(x; q, a)$ \textcolor{black}{signifies} the number of primes $p \leq x$ such that $p \equiv a \tpmod q$, then the inequality $\pi(x; 4, 3) \geq \pi(x; 4, 1)$ holds for more than 97~\% of $x < 10^{11}$. On the other hand, were we to leverage Dirichlet's theorem~on~primes in arithmetic progressions, then it \textcolor{black}{is} expected that the number of the primes of the form~$4k+1$ and $4k+3$ should be asymptotically equal. Therefore, the Chebyshev~bias indicates that the primes of the form $4k+3$ appear earlier than those of the form $4k+1$. \textcolor{black}{Classical triumphs include the work of} Littlewood~\cite{Littlewood1914} \textcolor{black}{who} established that $\pi(x; 4, 3)-\pi(x; 4, 1)$ changes its sign infinitely often. \textcolor{black}{Knapowski--Tur\'{a}n~\cite{KnapowskiTuran1962} conjectured that the density of the numbers $x$~for~which $\pi(x; 4, 3) \geq \pi(x; 4, 1)$ holds is $1$. Nonetheless, the work of Kaczorowski~\cite{Kaczorowski1995} shows the falsity of their conjecture conditionally on the Generalised Riemann Hypothesis. Note that they have a logarithmic density via work of Rubinstein--Sarnak~\cite{RubinsteinSarnak1994}, which is approximately $0.9959 \cdots$.}

In this fundamental scenario, the next layer of methodological depth came with the introduction of a weighted counting function that allows one to scrutinise the above phenomenon. \textcolor{black}{The work of} Aoki--Koyama~\cite{AokiKoyama2023} \textcolor{black}{introduces} the counting function
\begin{equation}\label{weighted-counting-function}
\pi_{s}(x; q, a) \coloneqq \sum_{\substack{p \leq x \\ p \equiv a \tpmod q}} \frac{1}{p^{s}}, \qquad s \geq 0,
\end{equation}
extending $\pi(x; q, a) = \pi_{0}(x; q, a)$, where the smaller prime $p$ permits a higher contribution to $\pi_{s}(x; q, a)$ as long as we fix $s > 0$. The function $\pi_{s}(x; q, a)$ ought to be more appropriate~than $\pi(x; q, a)$ to represent the phenomenon, since it reflects the size of the primes that $\pi(x; q, a)$ ignores. Although the natural density of the set $\{x > 0 \mid \pi_{s}(x; 4, 3)-\pi_{s}(x; 4, 1) > 0 \}$~does~not exist when $s = 0$ under the Generalised Riemann Hypothesis (see the article~\cite{Kaczorowski1995}), they have shown under~the Deep Riemann Hypothesis (DRH) that it~would be equal to $1$ when $s = 1/2$. \textcolor{black}{More precisely}, the Chebyshev bias could~be realised in terms of the asymptotic formula
\begin{equation}\label{cheby}
\pi_{\frac{1}{2}}(x; 4, 3)-\pi_{\frac{1}{2}}(x; 4, 1) = \frac{1}{2} \log \log x+c+o(1)
\end{equation}
 as $x \to \infty$, where $c$ is a constant.

The Chebyshev~biases for prime ideals $\frak{p}$ of a global~field~$K$ \textcolor{black}{are formulated as~follows}.
\begin{definition}[Aoki--Koyama~\cite{AokiKoyama2023}]\label{definition-1}
Let $c(\mathfrak{p}) \in \R$ be a sequence over primes $\mathfrak{p}$ of $K$ such~that
\begin{equation*}
\lim_{x \to \infty} \frac{\#\{\mathfrak{p} \mid c(\mathfrak{p}) > 0, \mathrm{N}(\mathfrak{p}) \leq x \}}
{\#\{\mathfrak{p} \mid c(\mathfrak{p}) < 0, \mathrm{N}(\mathfrak{p}) \leq x \}} = 1.
\end{equation*}
We say that $c(\mathfrak{p})$ has a \textit{Chebyshev bias towards being positive} if there exists a constant $C > 0$ such that
\begin{equation*}
\sum_{\mathrm{N}(\mathfrak{p}) \leq x} \frac{c(\mathfrak{p})}{\sqrt{\mathrm{N}(\mathfrak{p})}} \sim C \log \log x,
\end{equation*}
where $\mathfrak{p}$ ranges over primes of $K$. On the other hand, we say that $c(\mathfrak{p})$ is \textit{unbiased} if
\begin{equation*}
\sum_{\mathrm{N}(\mathfrak{p}) \leq x} \frac{c(\mathfrak{p})}{\sqrt{\mathrm{N}(\mathfrak{p})}} = O(1).
\end{equation*}
\end{definition}

\begin{definition}[Aoki--Koyama~\cite{AokiKoyama2023}]\label{definition-2}
Assume that the set of all primes $\mathfrak{p}$ of $K$ with $\mathrm{N}(\mathfrak{p}) \leq x$ is expressed as a disjoint union $P_{1}(x) \cup P_{2}(x)$ and that their proportion converges to
\begin{equation*}
\delta = \lim_{x \to \infty} \frac{|P_{1}(x)|}{|P_{2}(x)|}.
\end{equation*}
We say that there exists a Chebyshev bias towards $P_{1}$ or a Chebyshev bias against $P_{2}$ if
\begin{equation*}
\sum_{p \in P_{1}(x)} \frac{1}{\sqrt{\mathrm{N}(\mathfrak{p})}}
 - \delta \sum_{p \in P_{2}(x)} \frac{1}{\sqrt{\mathrm{N}(\mathfrak{p})}}
\sim C \log \log x
\end{equation*}
for some $C > 0$. On the other hand, we say that there exist no biases between $P_{1}$ and $P_{2}$ if
\begin{equation*}
\sum_{p \in P_{1}(x)} \frac{1}{\sqrt{\mathrm{N}(\mathfrak{p})}}
 - \delta \sum_{p \in P_{2}(x)} \frac{1}{\sqrt{\mathrm{N}(\mathfrak{p})}} = O(1).
\end{equation*}
\end{definition}

\textcolor{black}{Definitions~\ref{definition-1} and~\ref{definition-2} differ from those in~\cite{AkbaryNgShahabi2014,Devin2020,RubinsteinSarnak1994} and formulate an asymptotic formula for the size of the discrepancy caused by Chebyshev's bias disregarded in the conventional definitions of the length of the interval in terms of the limiting distributions. Definitions~\ref{definition-1} and~\ref{definition-2} involve little information on the density distributions and both types of formulations appear not to have any logical connections. They shed some light on Chebyshev's bias from different directions.}

It is \textcolor{black}{shown} in~\cite[Corollary 3.2]{AokiKoyama2023} that the Chebyshev bias~\eqref{cheby} against the quadratic residue $1 \tpmod 4$ is equivalent to the convergence of the Euler product at the centre $s = 1/2$ for the Dirichlet $L$-function $L(s, \chi_{-4})$ with $\chi_{-4}$ denoting the nontrivial character modulo $4$. This is \textcolor{black}{related} to DRH proposed by Kurokawa~\cite{KanekoKoyamaKurokawa2022,Kurokawa2012} in 2012.

Aoki--Koyama~\cite{AokiKoyama2023} have observed various instances of Chebyshev biases that certain prime ideals in Galois extensions of global fields have over others, with an emphasis on the biases against splitting primes and principal prime ideals. Koyama--Kurokawa~\cite{KoyamaKurokawa2022} have obtained~an analogue of this phenomenon for Ramanujan's $\tau$-function $\tau(p)$ and shown that DRH for the automorphic $L$-function $L(s+11/2, \Delta)$ implies the bias of $\tau(p)/p^{11/2}$ towards positive values.

In this work, we delve into such phenomena on the prime number races that elliptic~curves over function fields give rise to. Let $E$ be an elliptic curve over $K$ and let $E_{v}$ be the reduction of $E$ on the residue field $k_{v}$ at a finite place $v$ of $K$. If $E$ has good reduction at $v$, we define
\begin{equation*}
a_{v} = a_{v}(E) \coloneqq q_{v}+1-\#E_{v}(k_{v}),
\end{equation*}
where $q_{v} = \#k_{v}$ and $\#E_{v}(k_{v})$ is the number of $k_{v}$-rational points on $E_{v}$. The symbol $a_{v}$ can be extended to all other finite places $v$ as follows:
\begin{equation*}
a_{v} \coloneqq 
	\begin{cases}
	1 & \text{if $E$ has split multiplicative reduction at $v$},\\
	-1 & \text{if $E$ has nonsplit multiplicative reduction at $v$},\\
	0 & \text{if $E$ has additive reduction at $v$}.
\end{cases}
\end{equation*}
\textcolor{black}{This work aims} to prove the following asymptotic \textcolor{black}{corresponding to the case of $s = 1/2$~in~\eqref{weighted-counting-function}}.
\begin{theorem}\label{th2}
Assume $\mathrm{char}(K) > 0$ and that $E$ is a non-constant elliptic curve over $K$ in the \textcolor{black}{terminology} of Ulmer~\cite[Definitions 1.1.4]{Ulmer2011}. If $\mathrm{rank}(E) > 0$, then the sequence $a_{v}/\sqrt{q_{v}}$ has a Chebyshev bias towards being negative. To be more precise, we have that
\begin{equation}\label{biascor1}
\sum_{q_{v} \leq x} \frac{a_{v}}{q_{v}} = \textcolor{black}{\left(\frac{1}{2}-\mathrm{rank}(E) \right)} \log \log x+O(1).
\end{equation}
\end{theorem}

The proof uses the convergence of the Euler product at the centre, which follows from~DRH over function fields due to Conrad~\cite{Conrad2005} and Kaneko--Koyama--Kurokawa~\cite{KanekoKoyamaKurokawa2022}; see~\S\ref{deep-Riemann-hypothesis} for details.

\section*{Acknowledgements}\label{acknowledgements}
The authors would like to thank Nobushige Kurokawa for instructive discussions. The~first author acknowledges the support of the Masason Foundation and the Spirit of Ramanujan STEM Talent Initiative. The second author was partially supported by the INOUE ENRYO Memorial Grant 2022, TOYO University.

\section{Deep Riemann Hypothesis}\label{deep-Riemann-hypothesis}
Let $K$ be a one-dimensional global field that is either a number field or a function field in one variable over a finite field. For a place $v$ of $K$, let $M(v)$ denote a unitary matrix~of~degree $r_{v} \in \mathbb{N}$. We consider an $L$-function expressed as the Euler product
\begin{equation}\label{EP}
L(s, M) = \prod_{v < \infty} \det(1-M(v) q_{v}^{-s})^{-1},
\end{equation}
where  $q_{v}$ is the cardinal of the residue field $k_{v}$ at $v$. The product~\eqref{EP} is absolutely convergent in $\Re(s) > 1$. In this article, we assume that $L(s, M)$ has an analytic continuation as an~entire function on $\mathbb{C}$ and a functional equation relating values at $s$ and $1-s$. Moreover, we set
\begin{equation*}
\delta(M) = -\ord_{s = 1} L(s, M^{2}),
\end{equation*}
where $\ord_{s = 1}$ signifies the order of the zero at $s = 1$. We here do not presuppose that $M$ is~a representation. The square $M^{2}$ is interpreted as an Adams operation. Note that since
\begin{equation*}
L(s, M^{2}) = \prod_{v < \infty} \det(1-M(v)^{2} q_{v}^{-s})^{-1} = \frac{L(s, \mathrm{Sym}^{2} M)}{L(s, \wedge^{2} M)},
\end{equation*}
we derive the expression
\begin{equation}\label{delta}
\delta(M) = -\ord_{s = 1} L(s, \mathrm{Sym}^{2} M)+\ord_{s = 1} L(s, \wedge^{2} M),
\end{equation}
where $\mathrm{Sym}^{2}$ and $\wedge^{2}$ denote the symmetric and the exterior squares, respectively. If we~assume that $M$ is an Artin representation
\begin{equation*}
\rho \colon \mathrm{Gal}(K^{\mathrm{sep}}/K) \to \mathrm{Aut}_{\mathbb{C}}(V), \qquad \rho \ne \mathbbm{1}
\end{equation*}
with a representation space $V$, then
\begin{equation*}
\delta(M) = \mathrm{mult}(\mathbbm{1}, \mathrm{Sym}^{2} \rho)-\mathrm{mult}(\mathbbm{1}, \wedge^{2} \rho),
\end{equation*}
where $\mathrm{mult}(\mathbbm{1}, \sigma)$ is the multiplicity of the trivial representation $\mathbbm{1}$ in $\sigma$.

We are now in a position to formulate DRH due to Kurokawa~\cite{KanekoKoyamaKurokawa2022,Kurokawa2012} in a general context.
\begin{conjecture}[Deep Riemann Hypothesis]\label{DRH}
Keep the assumptions and notation as above. Let $m = \ord_{s = 1/2} L(s, M)$. Then the limit
\begin{equation}\label{limit}
\lim_{x \to \infty} \left((\log x)^{m} \prod_{q_{v} \leq x} \det \left(1-M(v) q_{v}^{-\frac{1}{2}} \right)^{-1} \right)
\end{equation}
satisfies the following conditions:
\begin{description}
\item[DRH (A)] The limit~\eqref{limit} exists and is nonzero.
\item[DRH (B)] The limit~\eqref{limit} satisfies
\begin{equation*}
\lim_{x \to \infty} \left((\log x)^{m} \prod_{q_{v} \leq x} \det \left(1-M(v) q_{v}^{-\frac{1}{2}} \right)^{-1} \right)
 = \frac{\sqrt{2}^{\delta(M)}}{e^{m \gamma} m!} \cdot L^{(m)} \left(\frac{1}{2}, M \right).
\end{equation*}
\end{description}
\end{conjecture}
It is obvious that DRH (B) indicates DRH (A). Nonetheless, DRH (A) is still meaningful since it is essentially equivalent to the Chebyshev biases. The following instances clarify this situation; we refer the reader to the work of Aoki--Koyama~\cite{AokiKoyama2023} for more interesting examples.
\begin{example}[Aoki--Koyama~\cite{AokiKoyama2023}]
If $K = \Q$, we denote a prime number by $v = p$. If~$r_{p} = 1$ for any $p$ and $M$ is the nontrivial Dirichlet character modulo 4, namely $M(p) = \chi_{-4}(p)$,~then we have that $L(s, M) = L(s, \chi_{-4})$ and $\delta(M) = 1$. It is shown in their work~\cite{AokiKoyama2023} that DRH~(A) for $L(s,\chi_{-4})$ is equivalent to the original form of the Chebyshev bias~\eqref{cheby}.
\end{example}

\begin{example}[Koyama--Kurokawa~\cite{KoyamaKurokawa2022}]
Let $K = \Q$ and $r_{p} = 2$ for any $p$, and let $\tau(p) \in \Z$ be Ramanujan's $\tau$-function defined for $q = e^{2\pi iz}$ with $\mathrm{Im}(z) > 0$ by
\begin{equation*}
\Delta(z) = q \prod_{k = 1}^{\infty} (1-q^{k})^{24} = \sum_{n = 1}^{\infty} \tau(n) q^{n}.
\end{equation*}
We introduce $M(p) = \begin{pmatrix} e^{i\theta_{p}} & 0 \\ 0 & e^{-i\theta_{p}} \end{pmatrix}$, where the Satake parameter $\theta_{p} \in [0, \pi] \cong \mathrm{Conj}(\mathrm{SU}(2))$~is defined by $\tau(p) = 2p^{11/2} \cos(\theta_{p})$. It then follows that the associated $L$-function
\begin{equation*}
L(s, M) = \prod_{p}(1-2\cos(\theta_{p})p^{-s}+p^{-2s})^{-1}
\end{equation*}
satisfies a functional equation for $s \leftrightarrow 1-s$. Using the conventional notation of~Ramanujan's $L$-function
\begin{equation*}
L(s, \Delta) = \sum_{n=1}^{\infty} \frac{\tau(n)}{n^{s}},
\end{equation*}
which satisfies a functional equation for $s \leftrightarrow 12-s$, we have that $L(s, M) = L(s+11/2, \Delta)$ and $\delta(M) = -1$. It is shown in~\cite{KoyamaKurokawa2022} that DRH (A) for $L(s, M) = L(s+11/2, \Delta)$ implies that there exists a Chebyshev bias for the sequence $\tau(p)/p^{11/2}$ towards being~positive.
\end{example}

Conjecture~\ref{DRH} is known to hold when the characteristic is positive. The proof was given~by Conrad~\cite[Theorems 8.1 and 8.2]{Conrad2005} under the second moment hypothesis, and the full proof was offered by Kaneko--Koyama--Kurokawa~\cite[Theorem 5.5]{KanekoKoyamaKurokawa2022}. We record their result as follows.
\begin{theorem}[Kaneko--Koyama--Kurokawa~\cite{KanekoKoyamaKurokawa2022}]\label{DRHchar>0}
Conjecture~\ref{DRH} is valid for $\mathrm{char}(K) > 0$.
\end{theorem}

\section{Proof of Theorem~\ref{th2}}\label{proof-of-Theorem-1.3}
Let $E$ be an elliptic curve over a global field $K$ and let $a_{v}$ be the same as in the introduction. We define the parameter $\theta_{v} \in [0, \pi] \cong \mathrm{Conj}(\mathrm{SU}(2))$ by $a_{v} = 2 \sqrt{q_{v}} \cos(\theta_{v})$. If one writes
\begin{equation*}
r_{v} = 
	\begin{cases}
	2 & \text{if $v$ is good},\\
	1 & \text{if $v$ is bad},
	\end{cases}
\end{equation*}
and
\begin{equation*}
M(v) = M_{E}(v) = 
	\begin{cases}
	\begin{pmatrix} e^{i\theta_{v}} & 0 \\ 0 & e^{-i\theta_{v}} \end{pmatrix} & \text{if $v$ is good},\\
	a_{v} & \text{if $v$ is bad},
	\end{cases}
\end{equation*}
then the $L$-function~\eqref{EP} is equal to
\begin{equation*}
L(s, M) = L(s, M_{E}) = \prod_{v \colon \text{good}} (1-2\cos(\theta_{v}) q_{v}^{-s}+q_{v}^{-2s})^{-1} 
\prod_{v \colon \text{bad}} (1-a_{v} q_{v}^{-s})^{-1}.
\end{equation*}
This Euler product is absolutely convergent in $\Re(s) > 1$ and has an meromorphic continuation to the whole complex plane $\mathbb{C}$ along with a functional equation for $s \leftrightarrow 1-s$.

The $L$-function $L(s, M_{E})$ is expressed in terms of an Artin-type $L$-function in the following manner. Fixing $\ell$ and $K$ with $\ell \ne \mathrm{char}(K)$, we obtain the representation
\begin{equation}\label{rho_{E}}
\rho_{E} \colon \mathrm{Gal}(K^{\mathrm{sep}}/K) \to \mathrm{Aut}(T_{\ell}(E) \otimes \Q_{\ell}),
\end{equation}
where $T_{\ell}(E) = \varprojlim\limits_{n} E[\ell^{n}]$ is the $\ell$-adic Tate module of $E/K$.
The Artin $L$-function associated to the Galois representation $\rho_{E}$ is defined by the Euler product
\begin{equation*}\label{artin}
L(s, \rho_{E}) = \prod_{v < \infty} \det(1-q_{v}^{-s} \rho_{E}(\mathrm{Frob}_{v} \vert V^{I_{v}}))^{-1}.
\end{equation*}
The Euler factors of $L(s, M_{E})$ are in accordance with those of $L(s+1/2, \rho_{E})$ for all places~$v$~at which $E$ has good reduction. In other words, the $L$-function $L(s, M_{E})$ equals $L(s+1/2, \rho_{E})$ up to finite factors stemming from bad places.

Conjecture~\ref{DRH} originates from the Birch--Swinnerton-Dyer conjecture in the following~form.
\begin{conjecture}[{Birch--Swinnerton-Dyer~\cite[Page~79~(A)]{BirchSwinnertonDyer1965}}]\label{BSD}
Let $E$ be an elliptic curve $E$~over $K$. Then there exists a constant $A > 0$ dependent on $E$ such that
\begin{equation}\label{BirchSwinnertonDyer}
\prod_{\substack{q_{v} \leq x \\ v \colon \text{good}}} \frac{\# E(k_{v})}{q_{v}} \sim A(\log x)^{r},
\end{equation}
where $r = \mathrm{rank}(E)$. Moreover, $r$ is equal to the order of vanishing of $L(s, M_{E})$ at $s = 1/2$.
\end{conjecture}
Since the left-hand side of~\eqref{BirchSwinnertonDyer} coincides with the Euler product over good places of~the $L$-function $L(s, M_{E})$ at $s = 1/2$, Conjecture~\ref{BSD} is indicative of DRH (A) for $L(s, M_{E})$.

\begin{theorem}\label{theorem1}
Keep the notation as above. The following conditions are equivalent.
\begin{enumerate}[(i)]
\item DRH (A) holds for $L(s, M)$.
\item There exists a constant $c$ such that
\begin{equation*}
\sum_{q_{v} \leq x} \frac{\mathrm{tr}(M(v))}{\sqrt{q_{v}}}
 = -\left(\frac{\delta(M)}{2}+m \right){\log \log x}+c+o(1),
\end{equation*}
where $m = \ord_{s = 1/2}L(s, M)$.
\end{enumerate}
\end{theorem}

\begin{proof}
We introduce
\begin{align*}
\text{I}(x) &= \sum_{q_{v} \leq x} \frac{\mathrm{tr}(M(v))}{\sqrt{q_{v}}},\\
\text{I\hspace{-.1mm}I}(x) &= \frac{1}{2} \sum_{q_{v} \leq x} \frac{\mathrm{tr}(M(v)^{2})}{q_{v}},\\
\text{I\hspace{-.1mm}I\hspace{-.1mm}I}(x) &= \sum_{k = 3}^{\infty} 
\frac{1}{k} \sum_{q_{v} \leq x} \frac{\mathrm{tr}(M(v)^{k})}{q_{v}^{k/2}}. 
\end{align*}
Since one has
\begin{equation*}
\text{I}(x)+\text{I\hspace{-.1mm}I}(x)+\text{I\hspace{-.1mm}I\hspace{-.1mm}I}(x)
 = \log \left(\prod_{q_{v} \leq x} \det \left(1-M(v) q_{v}^{-\frac{1}{2}} \right)^{-1} \right),
\end{equation*}
the condition (i) is equivalent to the claim that there exists a constant $L$ such that
\begin{equation}\label{K1}
m \log \log x+\text{I}(x)+\text{I\hspace{-.1mm}I}(x)+\text{I\hspace{-.1mm}I\hspace{-.1mm}I}(x) = L+o(1).
\end{equation}
The generalised Mertens theorem (see~\cite[Theorem 5]{Rosen1999} and~\cite[Lemma 5.3]{KanekoKoyamaKurokawa2022}) \textcolor{black}{gives}
\begin{equation}\label{K2}
\text{I\hspace{-.1mm}I}(x) = \frac{\delta(M)}{2} \log \log x+C_{1}+o(1)
\end{equation}
for some constant $C_{1}$. Additionally, it is straightforward to see that there exists a constant~$C_{2}$ such that
\begin{equation}\label{K3}
\text{I\hspace{-.1mm}I\hspace{-.1mm}I}(x) = C_{2}+o(1).
\end{equation}
Therefore, the estimates~\eqref{K1}--\eqref{K3} yield
\begin{equation*}
\text{I}(x) = -\left(\frac{\delta(M)}{2}+m \right) \log \log x+L-C_{1}-C_{2}+o(1).
\end{equation*}
If we assume (i), then the condition (ii) is valid with $c = L-C_{1}-C_{2}$. Conversely, if we~assume (ii), then~\eqref{K1} is valid with $L = c+C_{1}+C_{2}$. This finishes the proof of Theorem~\ref{theorem1}.
\end{proof}

In order to examine the asymptotic behavior of a sum over $v$ with $q_{v} \leq x$ as $x\to\infty$, it suffices to restrict ourselves to places $v$ at which $E$ has good reduction. When $v$ is good, the \textit{$n$-th symmetric power matrix} $\mathrm{Sym}^{n} M$ of size $n+1$ is given by
\begin{equation*}
(\mathrm{Sym}^{n} M)(v) = \mathrm{diag}(e^{in \theta_{v}},e^{i(n-2) \theta_{v}}, \cdots, e^{-i(n-2) \theta_{v}},e^{-in \theta_{v}}).
\end{equation*}
We calculate
\begin{equation*}
\mathrm{tr}(\mathrm{Sym}^{n} M)(v) = \frac{\sin((n+1) \theta_{v})}{\sin \theta_{v}}.
\end{equation*}
Extending the definition of $(\mathrm{Sym}^{n} M)(v)$ to all places $v$ by setting $(\mathrm{Sym}^{n} M)(v) = a_{v}^{n}$ for~bad places $v$, one defines the \textit{$n$-th symmetric power $L$-function} $L(s, \mathrm{Sym}^{n} M)$. With the~standard notation for the Galois representation $\rho = \rho_{E}$ in~\eqref{rho_{E}}, we have the following normalisation:
\begin{equation}\label{normalize}
L(s, \mathrm{Sym}^{n} M) = L \left (s+\frac{n}{2}, \mathrm{Sym}^{n} \rho \right).
\end{equation}
If $E$ is a non-constant elliptic curve in the \textcolor{black}{terminology} of Ulmer~\cite[Definitions 1.1.4]{Ulmer2011}, it~is known that $L(s, \mathrm{Sym}^{n} M)$ is a polynomial in $q^{-n/2-s}$ (see~\cite{ChaFiorilliJouve2016,Ulmer2005}) and the absolute values of its roots are equal to $q^{-(n+1)/2}$ with the normalisation~\eqref{normalize} taken into account. Therefore,~all zeroes of~\eqref{normalize} lie on the critical line $\Re(s) = 1/2$ and there holds
\begin{equation}\label{ord1}
\ord_{s = 1} L(s, \mathrm{Sym}^{n} M) = 0, \qquad n \in \mathbb{N}.
\end{equation}

\begin{lemma}\label{delta=-1}
If $\mathrm{char}(K) > 0$ and $E$ is a non-constant elliptic curve over $K$, then we have for $M = M_{E}$ that
$\delta(M) =  -1.$
\end{lemma}

\begin{proof}
Since $M$ is a unitary matrix of size 2, the exterior square matrix $\wedge^{2}M$ is trivial. Thus $\ord_{s = 1} L(s,\wedge^{2}M) = -1$. The claim follows from~\eqref{delta} and~\eqref{ord1}.
\end{proof}

In what follows, we use the shorthand $m_{n} = \ord_{s = 1/2} L(s, \mathrm{Sym}^{n} M)$.

\begin{theorem}\label{maintheorem}
If $\mathrm{char}(K) > 0$ and $E$ is a non-constant elliptic curve over $K$, then we have that
\begin{equation}\label{bias1}
\sum_{q_{v} \leq x} \frac{a_{v}}{q_{v}} = \textcolor{black}{\left(\frac{1}{2}-m_{1} \right)} \log \log x+O(1).
\end{equation}
In particular, if $\mathrm{rank}(E) > 0$, the sequence $a_{v}/\sqrt{q_{v}}$ has a Chebyshev bias towards being negative.
\end{theorem}

\begin{proof}
In~\cite[\S3.1.7]{Ulmer2005} and~\cite[Theorem 9.3]{Ulmer2011}, Ulmer proves that $L(s, M_{E})$ is a polynomial in $q^{-s}$ for any non-constant elliptic curve $E$ and hence it is entire and satisfies the assumption~of Conjecture~\ref{DRH}. By Theorem~\ref{DRHchar>0}, DRH holds for $L(s, M_{E})$. By Theorem~\ref{theorem1} and~Lemma~\ref{delta=-1}, we are able to prove~\eqref{bias1}. In order to justify the second assertion, we appeal to the work~of Ulmer~\cite[Theorem 12.1 (1)]{Ulmer2011}, which says that $\mathrm{rank}(E) \leq m_{1}$. Using this inequality, we~have $m_{1} \geq 1$ and thereby $C < 0$. This completes the proof of Theorem~\ref{maintheorem}.
\end{proof}

Theorem~\ref{maintheorem} is in accordance with the prediction of Sarnak~\cite[page 5]{Sarnak2007-2} that $\mathrm{rank}(E) > 0$ implies the existence of a bias towards being negative, although he considers the sequence $a_{v}$ instead of $a_{v}/\sqrt{q_{v}}$. He also states that $\mathrm{rank}(E) = 0$ implies the existence of a bias towards being positive. We ascertain this phenomenon under the Birch--Swinnerton-Dyer conjecture.

\begin{corollary}
Assume \textcolor{black}{that $L(1, M_{E}) \ne 0$}. If $\mathrm{char}(K) > 0$ and $E$ is a non-constant elliptic curve over $K$ such that $\mathrm{rank}(E) = 0$, then we have that
\begin{equation}
\sum_{q_{v} \leq x} \frac{a_{v}}{q_{v}} = \frac{1}{2} \log \log x+O(1).
\end{equation}
In other words, the sequence $a_{v}/\sqrt{q_{v}}$ has a Chebyshev bias towards being positive.
\end{corollary}

\begin{proof}
The Birch--Swinnerton-Dyer conjecture asserts that $m_{1} = \mathrm{rank}(E)$, which implies~that $m_{1} = 0$. Hence the asymptotic formula~\eqref{bias1} gives the desired result.
\end{proof}

\begin{remark}
Cha--Fiorilli--Jouve~\cite[Theorem 1.7]{ChaFiorilliJouve2016} have shown that there exist infinitely many elliptic curves $E/\F_q(T)$ such that the sequence $a_{v}$ is \textit{unbiased} in the following sense:
\begin{equation*}
\lim_{X \to \infty} \frac{1}{X} \sum_{\substack{x \leq X \\ T_{E}(x) > 0}} 1 = \frac{1}{2}
\end{equation*}
with
\begin{equation*}
T_{E}(x) = -\frac{x}{q^{x/2}} \sum_{\substack{\deg(v) \leq x \\ v \colon \text{good}}} 2 \cos \theta_{v}
 = -\frac{x}{q^{x/2}} \sum_{\substack{\deg(v) \leq x \\ v \colon \text{good}}} \frac{a_{v}}{q^{\deg(v)/2}}.
\end{equation*}
Despite their work discussing the Chebyshev bias for the sequence $a_{v}$, which is different from the sequence $a_{v}/\sqrt{q_{v}}$ dealt with in Theorem~\ref{maintheorem}, we \textcolor{black}{strongly believe} that these are constant elliptic curves whose $L$-functions have a pole at $s = 1$ and do not satisfy DRH.
\end{remark}

We obtain other types of biases on the Satake parameters $\theta_{v}$ by considering the symmetric square $L$-function.

\begin{theorem}\label{sym2}
If $\mathrm{char}(K) > 0$ and $E$ is a non-constant elliptic curve over $K$, then we have that
\begin{equation}\label{bias2}
\sum_{q_{v} \leq x} \frac{(2\cos \theta_{v}-1)(\cos \theta_{v}+1)}{\sqrt{q_{v}}}
= (1-m_{1}-m_{2}) \log \log x+O(1).
\end{equation}
In particular, if both $L(1/2, M) \ne 0$ and $L(1/2, \mathrm{Sym}^{2} M) \ne 0$ are true, the sequence $(2\cos \theta_{v}-1)(\cos \theta_{v}+1)$ has a Chebyshev bias towards being positive.
\end{theorem}

\begin{proof}
Applying Theorem~\ref{theorem1} to $L(s, M)$ and $L(s, \mathrm{Sym}^{2} M)$, we deduce
\begin{equation}\label{n=1}
\sum_{q_{v} \leq x} \frac{2 \cos \theta_{v}}{\sqrt{q_{v}}} = \left(\frac{1}{2}-m_{1} \right){\log \log x}+O(1)
\end{equation}
and
\begin{equation}\label{n=2}
\sum_{q_{v} \leq x} \frac{2 \cos 2 \theta_{v}}{\sqrt{q_{v}}}
 = \left(\frac{1}{2}-m_{2} \right){\log \log x}+O(1).
\end{equation}
It follows that $\eqref{n=2}+\eqref{n=1}$ equals
\begin{equation*}
\sum_{q_{v} \leq x} \frac{2(\cos \theta_{v}+\cos 2 \theta_{v})}{\sqrt{q_{v}}}
 = \sum_{q_{v} \leq x} \frac{2(2 \cos \theta_{v}-1)(\cos \theta_{v}+1)}{\sqrt{q_{v}}}
 = (1-m_{1}-m_{2}) \log \log x+O(1).
\end{equation*}
This concludes the proof of Theorem~\ref{sym2}.
\end{proof}

We give an example of unbiased sequences constructed from the Satake parameters $\theta_{v}$.

\begin{theorem}\label{unbiased}
If $\mathrm{char}(K) > 0$ and $E$ is a non-constant elliptic curve over $K$, then we have that
\begin{equation}\label{bias3}
\sum_{q_{v} \leq x} \frac{2(2\cos \theta_{v}+1)(\cos \theta_{v}-1)}{\sqrt{q_{v}}} = (m_{1}-m_{2}) \log \log x+O(1).
\end{equation}
In particular, if $m_{1} = m_{2}$, the sequence $(2\cos \theta_{v}+1)(\cos \theta_{v}-1)$ is unbiased.
\end{theorem}

\begin{proof}
It follows that $\eqref{n=2}-\eqref{n=1}$ equals
\begin{equation*}
\sum_{q_{v} \leq x} \frac{2(\cos2\theta_{v}-\cos \theta_{v})}{\sqrt{q_{v}}}
 = \sum_{q_{v} \leq x} \frac{2(2\cos \theta_{v}+1)(\cos \theta_{v}-1)}{\sqrt{q_{v}}}
=(m_{1}-m_{2}) \log \log x+O(1).
\end{equation*}
This concludes the proof of Theorem~\ref{unbiased}.
\end{proof}


\begin{thebibliography}{10}

\bibitem{AkbaryNgShahabi2014}
A.~Akbary, N.~Ng, and M.~Shahabi.
\newblock Limiting distributions of the classical error terms of prime number
  theory.
\newblock {\em Q. J. Math.}, 65(3):743--780, 2014.

\bibitem{AokiKoyama2023}
M.~Aoki and S.~Koyama.
\newblock {C}hebyshev's bias against splitting and principal primes in global
  fields.
\newblock {\em to appear in J. Number Theory}, 35 pages, 2023.
\newblock \url{https://doi.org/10.1016/j.jnt.2022.10.005}.

\bibitem{BirchSwinnertonDyer1965}
B.~J. Birch and H.~P.~F. Swinnerton-Dyer.
\newblock Notes on elliptic curves. {I\hspace{-.1mm}I}.
\newblock {\em J. Reine Angew. Math.}, 218:79--108, 1965.

\bibitem{ChaFiorilliJouve2016}
B.~Cha, D.~Fiorilli, and F.~Jouve.
\newblock Prime number races for elliptic curves over function fields.
\newblock {\em Ann. Sci. {\'{E}}cole Norm. Sup. (4)}, 49(5):1239--1277, 2016.

\bibitem{Conrad2005}
K.~Conrad.
\newblock Partial {E}uler products on the critical line.
\newblock {\em Canad. J. Math.}, 57(2):267--297, 2005.

\bibitem{Devin2020}
L.~Devin.
\newblock Chebyshev's bias for analytic {$L$}-functions.
\newblock {\em Math. Proc. Cambridge Philos. Soc.}, 169(1):103--140, 2020.

\bibitem{Kaczorowski1995}
J.~Kaczorowski.
\newblock On the distribution of primes (mod {$4$}).
\newblock {\em Analysis}, 15(2):159--171, 1995.

\bibitem{KanekoKoyamaKurokawa2022}
I.~Kaneko, S.~Koyama, and N.~Kurokawa.
\newblock Towards the {D}eep {R}iemann {H}ypothesis for {$\mathrm{GL}_{n}$}.
\newblock {\em arXiv e-prints}, 17 pages, 2022.
\newblock \url{https://arxiv.org/abs/2206.02612}.

\bibitem{KnapowskiTuran1962}
S.~Knapowski and P.~Tur{\'{a}}n.
\newblock Comparative prime-number theory. {I}. {I}ntroduction.
\newblock {\em Acta Math. Acad. Sci. Hungar.}, 13:299--314, 1962.

\bibitem{KoyamaKurokawa2022}
S.~Koyama and N.~Kurokawa.
\newblock {C}hebyshev's bias for {R}amanujan's {$\tau$}-function via the {D}eep
  {R}iemann {H}ypothesis.
\newblock {\em Proc. Japan Acad. Ser. A Math. Sci.}, 98(6):35--39, 2022.

\bibitem{Kurokawa2012}
N.~Kurokawa.
\newblock {\em The pursuit of the {R}iemann {H}ypothesis (in Japanese)}.
\newblock Gijutsu Hyouron-sha, Tokyo, 2012.

\bibitem{Littlewood1914}
J.~E. Littlewood.
\newblock Sur la distribution des nombres premiers.
\newblock {\em C. R. Acad. Sci. Paris}, 158:1869--1872, 1914.

\bibitem{Rosen1999}
M.~I. Rosen.
\newblock A generalization of {M}ertens' theorem.
\newblock {\em J. Ramanujan Math. Soc.}, 14:1--19, 1999.

\bibitem{RubinsteinSarnak1994}
M.~Rubinstein and P.~Sarnak.
\newblock {C}hebyshev's bias.
\newblock {\em Experiment. Math.}, 3(3):173--197, 1994.

\bibitem{Sarnak2007-2}
P.~Sarnak.
\newblock Letter to: {B}arry {M}azur on ``{C}hebyshev's bias'' for {$\tau(p)$},
  2007.
\newblock
  \url{https://publications.ias.edu/sites/default/files/MazurLtrMay08.PDF}.

\bibitem{Ulmer2005}
D.~Ulmer.
\newblock Geometric non-vanishing.
\newblock {\em Invent. Math.}, 159(1):133--186, 2005.

\bibitem{Ulmer2011}
D.~Ulmer.
\newblock Elliptic curves over function fields.
\newblock In C.~Popescu, K.~Rubin, and A.~Silverberg, editors, {\em Arithmetic
  of $L$-Functions}, volume~18 of {\em IAS/Park City Mathematics Series}, pages
  211--280. Amer. Math. Soc., Providence, RI, 2011.

\end{thebibliography}

\end{document}